\documentclass[english]{amsart}
\usepackage[T1]{fontenc}
\usepackage[latin9]{inputenc}
\usepackage{fancyhdr}
\usepackage{babel}
\usepackage{amsthm}
\usepackage{amstext}
\usepackage{amssymb}
\usepackage{fixltx2e}
\usepackage[all]{xy}
\usepackage[unicode=true,
 bookmarks=true,bookmarksnumbered=false,bookmarksopen=false,
 breaklinks=false,pdfborder={0 0 1},backref=false,colorlinks=false]
 {hyperref}
\hypersetup{pdftitle={Centralizer of quantum trace},
 pdfauthor={Szabolcs Mészáros},
 pdfsubject={Quantum matrices}}

\makeatletter
\numberwithin{equation}{section}
\numberwithin{figure}{section}
\theoremstyle{plain}
\newtheorem{thm}{\protect\theoremname}[section]
  \theoremstyle{plain}
  \newtheorem{prop}[thm]{\protect\propositionname}
  \theoremstyle{plain}
  \newtheorem{lem}[thm]{\protect\lemmaname}
  \theoremstyle{remark}
  \newtheorem{rem}[thm]{\protect\remarkname}

\usepackage{ifpdf} % part of the hyperref bundle
\ifpdf % if pdflatex is used

 \IfFileExists{lmodern.sty}{\usepackage{lmodern}}{}

\fi % end if pdflatex is used

\usepackage[all]{xy}
\usepackage{enumitem}
\usepackage{indentfirst}

\AtBeginDocument{

}

\makeatother

  \providecommand{\lemmaname}{Lemma}
  \providecommand{\propositionname}{Proposition}
  \providecommand{\remarkname}{Remark}
\providecommand{\theoremname}{Theorem}

\begin{document}

\lhead{Cocommutative elements in quantum matrices}

\rhead{ Szabolcs Mészáros}

\chead{}

\title{Cocommutative elements form a maximal commutative subalgebra in quantum
matrices}

\author{Szabolcs Mészáros}
\begin{abstract}
In this paper we prove that the subalgebras of cocommutative elements
in the quantized coordinate rings of $M_{n}$, $GL_{n}$ and $SL_{n}$
are the centralizers of the trace $x_{1,1}+\dots+x_{n,n}$ in each
algebra, for $q\in\mathbb{C}^{\times}$ being not a root of unity.
In particular, it is not only a commutative subalgebra as it was known
before, but it is a maximal one. 
\end{abstract}

\maketitle

\section{Introduction}

In \cite{DL1} M. Domokos and T. Lenagan determined generators for
the subalgebra of cocommutative elements in the quantized coordinate
ring of the general linear group $\mathcal{O}_{q}\big(GL_{n}(\mathbb{C})\big)$
with $q$ being not a root of unity. Their proof was based on the
observation that these are exactly the invariants of some quantum
analogue of the conjugation action of $GL_{n}(\mathbb{C})$ on $\mathcal{O}\big(GL_{n}(\mathbb{C})\big)$
which may be called modified adjoint coaction. It turned out that
this ring of invariants is basically the same as in the classical
setting, namely it is a polynomial ring generated by the quantum versions
of the trace functions. In \cite{DL2} they proved that it is a more
general phenomenon: the subalgebra of cocommutative elements $\mathcal{O}_{q}(G)^{\mathrm{coc}}$
for the quantized coordinate ring $\mathcal{O}_{q}(G)$ of a simply-connected,
simple Lie group $G$ is always isomorphic to its classical counterpart
$\mathcal{O}(G)^{\mathrm{coc}}$, as a consequence of the Peter-Weyl
decomposition for quantized coordinate rings (see \cite{H,MNY}).
This way, they obtained generators for the $\mathcal{O}_{q}(G)^{\mathrm{coc}}$
subalgebras and for the related FRT-bialgebras. In the present paper,
however, we will discuss a property of $\mathcal{O}_{q}\big(GL_{n}(\mathbb{C})\big)$
that does not hold if $q=1$ or if it is a root of unity.

The correspondence between $\mathcal{O}_{q}(G)^{\mathrm{coc}}$ and
$\mathcal{O}(G)^{\mathrm{coc}}$ does not stop on the level of their
algebra structure. In the case of $G=GL_{n}(\mathbb{C})$, Aizenbud
and Yacobi in \cite{AY} proved the quantum analog of Kostant's theorem
stating that $\mathcal{O}_{q}\big(M_{n}(\mathbb{C})\big)$ is a free
module over the ring of invariants under the adjoint coaction of $\mathcal{O}_{q}\big(GL_{n}(\mathbb{C})\big)$,
provided that $q$ is not a root of unity. Hence, the description
of $\mathcal{O}_{q}\big(GL_{n}(\mathbb{C})\big)$ as a module over
$\mathcal{O}_{q}\big(GL_{n}(\mathbb{C})\big)^{\mathrm{coc}}$ is available.
The classical theorem of Kostant can be interpreted as the $q=1$
case of this result. These type of statements (see \cite{B,JL}) can
also be used as tools to obtain other results, as in \cite{Y1} the
Joseph localizations being free over certain subalgebras is proved
and applied to establish numerous results, including a description
of the maximum spectra of $\mathcal{O}_{q}(G)$.

In this paper, we further investigate the relation of the subalgebra
$\mathcal{O}_{q}\big(GL_{n}(\mathbb{C})\big)^{\mathrm{coc}}$ to the
whole algebra $\mathcal{O}_{q}\big(GL_{n}(\mathbb{C})\big)$ when
$q$ is not a root of unity. Namely, we prove the following theorem: 
\begin{thm}
For $n\in\mathbb{N}^{+}$ and $q\in\mathbb{C}^{\times}$ not a root
of unity, the subalgebra of cocommutative elements is a maximal commutative
$\mathbb{C}$-subalgebra in $\mathcal{O}_{q}\big(M_{n}(\mathbb{C})\big)$,
$\mathcal{O}_{q}\big(GL_{n}(\mathbb{C})\big)$ and $\mathcal{O}_{q}\big(SL_{n}(\mathbb{C})\big)$.\label{thm:The-subalgebra}
\end{thm}
By Theorem 6.1 in \cite{DL1}, these subalgebras are determined by
certain pairwise commuting sums of (principal) quantum minors (denoted
by $\sigma_{i}$, $i=1,\dots,n$) that are defined in Section \ref{sec:Prerequisites}.
It means that it is enough to prove that the intersection of the centralizers
of these explicit commuting generators is not bigger than their generated
subalgebra. So we prove the following (stronger) statement:
\begin{thm}
For $n\in\mathbb{N}^{+}$ and $q\in\mathbb{C}^{\times}$ not a root
of unity, the centralizer of $\sigma_{1}=x_{11}+\dots+x_{nn}$ in
$\mathcal{O}_{q}\big(M_{n}(\mathbb{C})\big)$ (resp. $\sigma_{1}\in\mathcal{O}_{q}\big(GL_{n}(\mathbb{C})\big)$
and $\overline{\sigma}_{1}\in\mathcal{O}_{q}\big(SL_{n}(\mathbb{C})\big)$)
as a unital $\mathbb{C}$-subalgebra is generated by
\begin{itemize}
\item $\sigma_{1},\dots,\sigma_{n-1},\sigma_{n}$ in the case of $\mathcal{O}_{q}\big(M_{n}(\mathbb{C})\big)$,
\item $\sigma_{1},\dots\sigma_{n-1},\sigma_{n},\sigma_{n}^{-1}$ in the
case of $\mathcal{O}_{q}\big(GL_{n}(\mathbb{C})\big)$, and 
\item $\overline{\sigma}_{1},\dots\overline{\sigma}_{n-1}$ in the case
of $\mathcal{O}_{q}\big(SL_{n}(\mathbb{C})\big)$. \label{thm:The-centralizer-of-the-trace}
\end{itemize}
\end{thm}
It is important to note that, while the theorems in \cite{DL1,DL2}
are quantum analogues of theorems established in the commutative case
and they are also true if $q$ is a root of unity (see \cite{AZ}),
this result, however, has no direct commutative counterpart and also
fails if $q$ is a root of unity since then the algebras have large
center. 

In Ore extensions of polynomials rings or in lower Gelfand-Kirillov
dimension, it is not a rare phenomenon that a centralizer of an element
$a\in A$ is commutative but larger than $\mathbb{C}\langle a,Z(A)\rangle$,
see \cite{BS,RS}. The above investigation shows that it also occurs
in less regular situations for some very special elements in quantized
function algebras. 

As a consequence of Theorem \ref{thm:The-subalgebra}, we can get
other maximal commutative subalgebras by applying automorphisms. One
of these automorphic images is the invariants of the adjoint coaction,
as it is discussed in Remark \ref{rem:final remark automorphism}.
Moreover, by an analogous argument as we use in the proof of Theorem
\ref{thm:The-subalgebra}, it is possible to find maximal Poisson-commutative
subalgebras in the semi-classical limits. We will discuss these issues
in a subsequent paper.

The article is organized as follows: In the next section we introduce
the relevant notions and notations. In Section \ref{sec:main}, first
we prove Proposition \ref{prop:implications} stating that it is enough
to prove Theorem \ref{thm:The-centralizer-of-the-trace} for any of
the three algebras $\mathcal{O}_{q}\big(M_{n}(\mathbb{C})\big)$,
$\mathcal{O}_{q}\big(GL_{n}(\mathbb{C})\big)$ or $\mathcal{O}_{q}\big(SL_{n}(\mathbb{C})\big)$.
Then, in Section \ref{sec:Case-of n=00003D2} we discuss the proof
of case $n=2$ as a starting step of the induction used to prove Theorem
\ref{thm:The-centralizer-of-the-trace}. Finally, in Section \ref{sec:Proof}
we prove the induction step to complete the proof of the theorem.

\section{Preliminaries\label{sec:Prerequisites}}

\subsection{Quantized coordinate rings}

Assume that $n\in\mathbb{N}^{+}$ and $q\in\mathbb{C}^{\times}$ is
not a root of unity. Define $\mathcal{O}_{q}\big(M_{n}(\mathbb{C})\big)$,
the quantized coordinate ring of $n\times n$ matrices as the unital
$\mathbb{C}$-algebra generated by the $n^{2}$ generators $x_{i,j}$
for $1\leq i,j\leq n$ that are subject to the following relations:
\[
x_{i,j}x_{k,l}=\begin{cases}
x_{k,l}x_{i,j}+(q-q^{-1})x_{i,l}x_{k,j} & \textrm{if }i<k\textrm{ and }j<l\\
qx_{k,l}x_{i,j} & \textrm{if }(i=k\textrm{ and }j<l)\textrm{ or }(j=l\textrm{ and }i<k)\\
x_{k,l}x_{i,j} & \textrm{otherwise}
\end{cases}
\]
for all $1\leq i,j,k,l\leq n$. It turns out to be a finitely generated
$\mathbb{C}$-algebra which is a Noetherian domain. (For a detailed
exposition, see \cite{BG}.) Furthermore, it can be endowed with a
coalgebra structure by setting $\varepsilon(x_{i,j})=\delta_{i,j}$
and $\Delta(x_{i,j})=\sum_{k=1}^{n}x_{i,k}\otimes x_{k,j}$ turning
$\mathcal{O}_{q}\big(M_{n}(\mathbb{C})\big)$ into a bialgebra. 

Similarly, one can define the non-commutative deformations of the
coordinate rings of $GL_{n}$ and $SL_{n}$ using the quantum determinant

\[
\mathrm{det}_{q}:=\sum_{s\in S_{n}}(-q)^{\ell(s)}x_{1,s(1)}x_{2,s(2)}\dots x_{n,s(n)}
\]
where $\ell(\sigma)$ stands for the length of $\sigma$ in the Coxeter
group $S_{n}$. This definition can be ``legitimized'' by considering
the quantum exterior algebra $\Lambda_{q}(\mathbb{C}^{n})$ (see \cite{BG}).
Also its special behavior is justified by the fact that it is a group-like
element (i.e. $\Delta(\det_{q})=\det_{q}\otimes\det_{q}$) and it
generates the center of $\mathcal{O}_{q}\big(M_{n}(\mathbb{C})\big)$.
Then -- analogously to the classical case -- one defines 
\[
\mathcal{O}_{q}\big(SL_{n}(\mathbb{C})\big):=\mathcal{O}_{q}\big(M_{n}(\mathbb{C})\big)/(\mathrm{det}_{q}-1)\qquad\mathcal{O}_{q}\big(GL_{n}(\mathbb{C})\big):=\mathcal{O}_{q}\big(M_{n}(\mathbb{C})\big)\big[\mathrm{det}_{q}^{-1}\big]
\]
where inverting $\det_{q}$ cannot cause any problem because it is
central hence normal. The comultiplication and counit on $\mathcal{O}_{q}\big(M_{n}(\mathbb{C})\big)$
induce coalgebra structures on these algebras as well. In particular,
$\mathcal{O}_{q}\big(M_{n}(\mathbb{C})\big)$ is a subbialgebra of
$\mathcal{O}_{q}\big(GL_{n}(\mathbb{C})\big)$. In the case of $\mathcal{O}_{q}\big(SL_{n}(\mathbb{C})\big)$
and $\mathcal{O}_{q}\big(GL_{n}(\mathbb{C})\big)$ it is possible
to define antipodes that turn them into Hopf algebras.

\subsection{Quantum minors}

We call an element $a$ of a coalgebra $A$ cocommutative if $\Delta(a)=(\tau\circ\Delta)(a)$
where $\tau:A\otimes A\to A\otimes A$ is the flip $\tau(a\otimes b)=b\otimes a$.
Hence, we can define $A^{\mathrm{coc}}$, the subset of cocommutative
elements in $A$ which is necessarily a subalgebra if $A$ is a bialgebra.
For $A=\mathcal{O}_{q}\big(M_{n}(\mathbb{C})\big)$ the quantum determinant
is cocommutative since it is group-like. Moreover, by generalizing
the notion of $\det_{q}$, one can give an explicit description of
$A^{\mathrm{coc}}$ as it is proved in \cite{DL1}. For this purpose,
let us define the quantum minors for $I,J\subseteq\{1,\dots,n\}$,
$I=(i_{1},\dots,i_{t})$ and $J=(j_{1},\dots,j_{t})$ as
\[
[I\,|\, J]:=\sum_{s\in S_{t}}(-q)^{\ell(s)}x_{i_{1},j_{s(1)}}\dots x_{i_{t},j_{s(t)}}=\mathrm{det}_{q}\big(\mathbb{C}\langle x_{i,j}\ |\ i\in I,\ j\in J\rangle\big)\in A
\]
where $\mathbb{C}\langle\dots\rangle$ stands for the generated $\mathbb{C}$-subalgebra
and $\mathrm{det}_{q}\big(\mathbb{C}\langle x_{i,j}\ |\ i\in I,\ j\in J\rangle\big)$
denotes the quantum determinant of the subalgebra generated by $\{x_{i,j}\}_{i\in I,j\in J}$,
which can be identified with $\mathcal{O}_{q}\big(M_{t}(\mathbb{C})\big)$.
Now, one may compute 
\[
\Delta\big([I\,|\, J]\big)=\sum_{|K|=t}[I\,|\, K]\otimes[K\,|\, J]
\]
so we get cocommutative elements by taking
\[
\sigma_{i}=\sum_{|I|=i}[I\,|\, I]\in\mathcal{O}_{q}\big(M_{n}(\mathbb{C})\big)
\]
for all $1\leq i\leq n$. For $i=n$ we get $\det_{q}$ again and
in the case of $i=1$, it is $\sigma_{1}=x_{1,1}+x_{2,2}+\dots+x_{n,n}$. 

We will use $\sigma_{i}$ and $\overline{\sigma}_{i}$ for the induced
elements 
\[
\overline{\sigma}_{i}=\sigma_{i}+(\mathrm{det}_{q}-1)\in\mathcal{O}_{q}\big(SL_{n}(\mathbb{C})\big)\qquad\sigma_{i}\in\mathcal{O}_{q}\big(M_{n}(\mathbb{C})\big)\leq\mathcal{O}_{q}\big(GL_{n}(\mathbb{C})\big)
\]
an we will write $\sigma_{i}(A)$ for $\sigma_{i}$ in an algebra
$A$ isomorphic to $\mathcal{O}_{q}\big(M_{t}(\mathbb{C})\big)$ for
some $t$. Theorem 6.1 in \cite{DL1} states that the subalgebra of
cocommutative elements in $\mathcal{O}_{q}\big(M_{n}(\mathbb{C})\big)$
is freely generated (as a commutative algebra) by $\sigma_{1},\dots,\sigma_{n}$
, and (consequently) in $\mathcal{O}_{q}\big(GL_{n}(\mathbb{C})\big)$
it is generated by $\sigma_{1},\dots,\sigma_{n},\sigma_{n}^{-1}$,
giving an algebra isomorphic to $\mathbb{C}[t_{1},\dots,t_{n},t_{n}^{-1}]$.
The case of $SL_{n}$ is proved in \cite{DL2}: $\mathcal{O}_{q}\big(SL_{n}(\mathbb{C})\big)^{\mathrm{coc}}$
is generated by $\overline{\sigma}_{1},\dots,\overline{\sigma}_{n-1}$
and is isomorphic to $\mathbb{C}[t_{1},\dots,t_{n-1}]$.

\subsection{Poincaré-Birkhoff-Witt basis in the quantized coordinate ring of
matrices}

Several properties of $\mathcal{O}_{q}\big(M_{n}(\mathbb{C})\big)$
can be deduced by the observation that it is an iterated Ore extension.
It means that there exists a finite sequence of $\mathbb{C}$-algebras
$R_{0},R_{1},\dots,R_{n^{2}}$ such that $R_{0}=\mathbb{C}$ and $R_{i+1}=R_{i}[x_{i};\tau_{i},\delta_{i}]$,
the skew polynomial ring in $x_{i}$ for an appropriate automorphism
$\tau_{i}\in\mathrm{Aut}(R_{i})$ and a derivation $\delta_{i}\in\mathrm{Der}(R_{i})$. 

This choice of sequence of subalgebras includes an ordering on the
variables that is -- from the several possible options, now -- the
lexicographic ordering on $\{1,\dots,n\}\times\{1,\dots,n\}$. Moreover,
an iterated Ore extension as $\mathcal{O}_{q}\big(M_{n}(\mathbb{C})\big)$
has a Poincaré-Birkhoff-Witt basis, i.e. a $\mathbb{C}$-basis consisting
of the ordered monomials of the variables $x_{i,j}$. So, in the following,
we will refer to the following basis as the monomial basis of $\mathcal{O}_{q}\big(M_{n}(\mathbb{C})\big)$:
\[
x_{1,1}^{k_{1,1}}x_{1,2}^{k_{1,2}}x_{1,3}^{k_{1,3}}\dots x_{1,n}^{k_{1,n}}x_{2,1}^{k_{2,1}}\dots x_{n,n}^{k_{n,n}}\qquad\big(k_{ij}\in\mathbb{N},\, i,j\in\{1,2,\dots,n\}\big)
\]
It is indeed a basis, see \cite{BG}. 

Since the defining relations of $\mathcal{O}_{q}\big(M_{n}(\mathbb{C})\big)$
are homogeneous with respect to the total degree in the free algebra,
$\mathcal{O}_{q}\big(M_{n}(\mathbb{C})\big)$ inherits an $\mathbb{N}$-graded
algebra structure, i.e. 
\[
\mathcal{O}_{q}\big(M_{n}(\mathbb{C})\big)=\bigoplus_{d\in\mathbb{N}}\mathcal{O}_{q}\big(M_{n}(\mathbb{C})\big)_{d}
\]
as a vector space and $\mathcal{O}_{q}\big(M_{n}(\mathbb{C})\big)_{d}\cdot\mathcal{O}_{q}\big(M_{n}(\mathbb{C})\big)_{e}\subseteq\mathcal{O}_{q}\big(M_{n}(\mathbb{C})\big)_{d+e}$
for all $d,e\in\mathbb{N}$. Consequently, we may define a degree
function $\deg:\mathcal{O}_{q}\big(M_{n}(\mathbb{C})\big)\to\mathbb{N}$
as the maximum of the degrees of nonzero homogeneous components.

Although $\det_{q}-1$ is not homogeneous with respect to the total
degree, it is homogeneous modulo $n$ so the quotient algebra $\mathcal{O}_{q}\big(SL_{n}(\mathbb{C})\big)$
becomes a $\mathbb{Z}/n\mathbb{Z}$-graded algebra.

\subsection{Associated graded ring}

For a filtered ring $\big(R,\{\mathcal{F}^{d}\}_{d\in\mathbb{N}}\big)$
i.e. where $\{\mathcal{F}^{d}\}_{d\in\mathbb{N}}$ is an ascending
chain of subspaces in $R$ such that $R=\cup_{d\in\mathbb{N}}\mathcal{F}^{d}$
and $\mathcal{F}^{d}\cdot\mathcal{F}^{e}\subseteq\mathcal{F}^{d+e}$
for all $d,e\in\mathbb{N}$, we define its associated graded ring
\[
\mathrm{gr}(R):=\bigoplus_{d\in\mathbb{N}}\mathcal{F}^{d}/\mathcal{F}^{d-1}
\]
where we use the notation $\mathcal{F}^{-1}=\{0\}$. The multiplication
of $\mathrm{gr}(R)$ is defined in the usual way: 
\[
\mathcal{F}^{d}/\mathcal{F}^{d-1}\times\mathcal{F}^{e}/\mathcal{F}^{e-1}\to\mathcal{F}^{d+e}/\mathcal{F}^{d+e-1}
\]
\[
\big(x+\mathcal{F}^{d-1},y+\mathcal{F}^{e-1}\big)\mapsto xy+\mathcal{F}^{d+e-1}
\]
Clearly, it is a graded algebra by definition. In fact, $\mathrm{gr}(.)$
can be made into a functor defined as follows: for a morphism of filtered
algebras $f:\big(R,\{\mathcal{F}^{d}\}_{d\in\mathbb{N}}\big)\to\big(S,\{\mathcal{G}^{d}\}_{d\in\mathbb{N}}\big)$
(i.e. when $f(\mathcal{F}^{d})\subseteq\mathcal{G}^{d})$ we define
\[
\mathrm{gr}(f):\mathrm{gr}(R)\to\mathrm{gr}(S)
\]
\[
\big(x_{d}+\mathcal{F}^{d-1}\big)_{d\in\mathbb{N}}\mapsto\big(f(x_{d})+\mathcal{G}^{d-1}\big)_{d\in\mathbb{N}}
\]
One can check that it is indeed well defined and preserves composition.
A basic property of $\mathrm{gr}(.)$ is that if we have a map $f:R\to S$
such that $f(\mathcal{F}^{d})=\mathcal{G}^{d}$ then the $\mathrm{gr}(f)$
is also surjective.

\section{Equivalence of the statements\label{sec:main}}

As it is mentioned in the introduction, Theorem \ref{thm:The-subalgebra}
follows directly from Theorem \ref{thm:The-centralizer-of-the-trace}.
Indeed, since $\sigma_{i}$'s are commuting generators in the subalgebra
of cocommutative elements in $\mathcal{O}_{q}\big(M_{n}(\mathbb{C})\big)$,
$\mathcal{O}_{q}\big(GL_{n}(\mathbb{C})\big)$ and $\mathcal{O}_{q}\big(SL_{n}(\mathbb{C})\big)$
(see Section \ref{sec:Prerequisites}), any commutative subalgebra
containing the subalgebra of cocommutative elements is contained in
the centralizer of $\sigma_{1}$.

Moreover, the following proposition shows that it is enough to prove
Theorem \ref{thm:The-centralizer-of-the-trace} in the case of $\mathcal{O}_{q}\big(M_{n}(\mathbb{C})\big)$.
\begin{prop}
Assume that $n\in\mathbb{N}^{+}$ and $q\in\mathbb{C}^{\times}$ is
not a root of unity. The following are equivalent:\label{prop:implications}
\begin{enumerate}
\item The centralizer of $\sigma_{1}\in\mathcal{O}_{q}\big(M_{n}(\mathbb{C})\big)$
is generated by $\sigma_{1},\dots,\sigma_{n-1},\sigma_{n}$.
\item The centralizer of $\sigma_{1}\in\mathcal{O}_{q}\big(GL_{n}(\mathbb{C})\big)$
is generated by $\sigma_{1},\dots,\sigma_{n-1},\sigma_{n},\sigma_{n}^{-1}$.
\item The centralizer of $\overline{\sigma}_{1}\in\mathcal{O}_{q}\big(SL_{n}(\mathbb{C})\big)$
is generated by $\overline{\sigma}_{1},\dots,\overline{\sigma}_{n-1}$.
\end{enumerate}
\end{prop}
For the proof, we need the following short lemma:
\begin{lem}
Let $R=\oplus_{i\geq0}R_{i}$ be an $\mathbb{N}$-graded algebra and
$r\in R_{k}$ a central element that is not a zero-divisor. Then for
all $d\in\mathbb{N}$, $(r-1)\cap R_{d}=0$.\label{lem:no homogeneous element}\end{lem}
\begin{proof}
Since $r-1$ is central, its generated ideal is its generated left
ideal so $0\neq x\in(r-1)$ means that $x=y\cdot(r-1)$ for some $y\in R$.
Let $y=\sum_{i=i_{0}}^{\deg y}y_{i}\in\oplus_{i}R_{i}$ be the homogeneous
decomposition of $y$ where $y_{i_{0}}\neq0$. Then the highest degree
nonzero homogeneous component of $y\cdot(r-1)$ is $y_{\deg y}r$
which is of degree $\deg y+k$ since $r$ is not a zero-divisor. While
the lowest degree nonzero component of $y\cdot(r-1)$ is $-y_{i_{0}}$
which is of degree $i_{0}\leq\deg y<\deg y+k$. Therefore, $x$ cannot
be homogeneous.
\end{proof}

\begin{proof}[Proof of Proposition \ref{prop:implications}]
 Assume that $1)$ is true and let $h\in\mathcal{O}_{q}\big(GL_{n}(\mathbb{C})\big)$
that commutes with $\sigma_{1}$. By the definition of $\mathcal{O}_{q}\big(GL_{n}(\mathbb{C})\big)$,
there exists an $k\in\mathbb{N}$ such that $h\cdot\mathrm{det}_{q}^{k}\in\mathcal{O}_{q}\big(M_{n}(\mathbb{C})\big)\leq\mathcal{O}_{q}\big(GL_{n}(\mathbb{C})\big)$
which also commutes with $\sigma_{1}$ since $\mathrm{det}_{q}$ is
central. Therefore, by 1) we have $h\cdot\mathrm{det}_{q}^{k}\in\mathbb{C}\langle\sigma_{1},\dots,\sigma_{n-1},\sigma_{n}\rangle$
hence $h=h\cdot\mathrm{det}_{q}^{k}\cdot\det_{q}^{-k}\in\mathbb{C}\langle\sigma_{1},\dots,\sigma_{n-1},\sigma_{n},\sigma_{n}^{-1}\rangle$
and so $2)$ follows. Conversely, assume $2)$ and take an $h\in\mathcal{O}_{q}\big(M_{n}(\mathbb{C})\big)\subseteq\mathcal{O}_{q}\big(GL_{n}(\mathbb{C})\big)$
that commutes with $\sigma_{1}$. By the assumption, $h\in\mathbb{C}\langle\sigma_{1},\dots,\sigma_{n-1},\sigma_{n},\sigma_{n}^{-1}\rangle$
hence it is cocommutative in $\mathcal{O}_{q}\big(GL_{n}(\mathbb{C})\big)$.
Since $\mathcal{O}_{q}\big(M_{n}(\mathbb{C})\big)$ is a subbialgebra
of $\mathcal{O}_{q}\big(GL_{n}(\mathbb{C})\big)$, $h$ is cocommutative
in $\mathcal{O}_{q}\big(M_{n}(\mathbb{C})\big)$ too, hence by $\mathcal{O}_{q}\big(M_{n}(\mathbb{C})\big)^{\mathrm{coc}}=\mathbb{C}\langle\sigma_{1},\dots,\sigma_{n-1},\sigma_{n}\rangle$
(see Section \ref{sec:Prerequisites}) $1)$ follows.

Now, we prove $1)\iff3)$: First, assume $1)$ and let $\overline{h}\in\mathcal{O}_{q}\big(SL_{n}(\mathbb{C})\big)$
that commutes with $\overline{\sigma}_{1}$. Since $\mathcal{O}_{q}\big(SL_{n}(\mathbb{C})\big)$
is $\mathbb{Z}/n\mathbb{Z}$-graded and $\overline{\sigma}_{1}$ is
homogeneous with respect to this grading, its centralizer is generated
by homogeneous elements. So we may assume that $\overline{h}$ is
homogeneous as well. Let $k=\deg(\overline{h})$. Take an $h\in\mathcal{O}_{q}\big(M_{n}(\mathbb{C})\big)$
that represents $\overline{h}\in\mathcal{O}_{q}\big(SL_{n}(\mathbb{C})\big)=\mathcal{O}_{q}\big(M_{n}(\mathbb{C})\big)/(\det_{q}-1)$.
Let $h=\sum_{j=0}^{d}h_{jn+k}$ be the $\mathbb{N}$-homogeneous decomposition
of $h$ where $h_{jn+k}$ is homogeneous of degree $jn+k$ for all
$j\in\mathbb{N}$. (We do not need the other homogeneous components
as $\overline{h}$ is $\mathbb{Z}/n\mathbb{Z}$-homogeneous so we
may assume that $h$ has nonzero homogeneous components only in degrees
$\equiv\deg(\overline{h})$ modulo $n$.) Then we can take 
\[
h':=\sum_{j=0}^{d}h_{jn+k}\cdot\mathrm{det}_{q}^{d-j}\in\mathcal{O}_{q}\big(M_{n}(\mathbb{C})\big)_{dn+k}
\]
which is a homogeneous element of degree $dn+k$ representing $\overline{h}$
in $\mathcal{O}_{q}\big(M_{n}(\mathbb{C})\big)$. Therefore, $\sigma_{1}h'-h'\sigma_{1}\in(\det_{q}-1)\cap\mathcal{O}_{q}\big(M_{n}(\mathbb{C})\big)_{dn+k+1}$
because $\overline{\sigma}_{1}\overline{h}-\overline{h}\overline{\sigma}_{1}=0\in\mathcal{O}_{q}\big(SL_{n}(\mathbb{C})\big)$
and $\sigma_{1}$ is homogeneous of degree $1$. By Lemma \ref{lem:no homogeneous element},
we get $(\det_{q}-1)\cap\mathcal{O}_{q}\big(M_{n}(\mathbb{C})\big)_{dn+k+1}=0$
meaning $\sigma_{1}h'=h'\sigma_{1}$. Then applying $1)$ gives $h'\in\mathbb{C}\langle\sigma_{1},\dots,\sigma_{n}\rangle$
hence $\overline{h}\in\mathbb{C}\langle\overline{\sigma_{1}},\dots,\overline{\sigma_{n-1}}\rangle$
as we claimed.

Conversely, assume $3)$ and let $h\in\mathcal{O}_{q}\big(M_{n}(\mathbb{C})\big)$
such that $\sigma_{1}h=h\sigma_{1}$. Since $\sigma_{1}$ is $\mathbb{N}$-homogeneous,
its centralizer is also generated by homogeneous elements so we may
assume that $h$ is homogeneous. Then we can take the image $\overline{h}$
of $h$ in $\mathcal{O}_{q}\big(SL_{n}(\mathbb{C})\big)$ which is
homogeneous with respect to the $\mathbb{Z}/n\mathbb{Z}$-grading
of $\mathcal{O}_{q}\big(SL_{n}(\mathbb{C})\big)$. Let $k=\deg(\overline{h})$.
By the assumption, $\overline{h}$ commutes with $\overline{\sigma}_{1}$
hence $\overline{h}\in\mathbb{C}\langle\overline{\sigma}_{1},\dots,\overline{\sigma}_{n-1}\rangle$
by $3)$. This decomposition of $\overline{h}$ can be lifted to $\mathcal{O}_{q}\big(M_{n}(\mathbb{C})\big)$
giving an element $s\in\mathbb{C}\langle\sigma_{1},\dots,\sigma_{n-1}\rangle$
such that $h-s\in(\det_{q}-1)$. As $\overline{h}$ was $\mathbb{Z}/n\mathbb{Z}$-homogeneous,
$s$ can also be chosen to be $\mathbb{Z}/n\mathbb{Z}$-homogeneous
since $\overline{\sigma}_{1},\dots,\overline{\sigma}_{n-1}$ are $\mathbb{Z}/n\mathbb{Z}$-homogeneous.
Let $d=\frac{1}{n}\big(\max(\deg h,\deg s)-k\big)$ and take $s=\sum_{j=0}^{d}s_{jn+k}$
the homogeneous decomposition of $s$. If $\deg(s)>\deg(h)$ then
let 
\[
h'=h\cdot\mathrm{det}_{q}^{\frac{1}{n}\big(\deg(s)-\deg(h)\big)}
\]
 so now $d=\deg(s)=\deg(h')$. (The exponent is an integer since $\deg(h)=\deg(s)$
modulo $n$.) Otherwise, let $h'=h$.

Now, the same way as in the proof of $1)\Rightarrow3)$, we can modify
$s$ as follows: Let 
\[
s':=\sum_{j=0}^{d}s_{jn+k}\cdot\mathrm{det}_{q}^{d-j}
\]
Then $s'\in\mathbb{C}\langle\sigma_{1},\dots,\sigma_{n}\rangle$,
it is $\mathbb{N}$-homogeneous of degree $nd+k$, and $s-s'\in(\det_{q}-1)$.
So $h'-s'=(h'-h)+(h-s)+(s-s')\in(\det_{q}-1)\cap\mathcal{O}_{q}\big(M_{n}(\mathbb{C})\big)_{nd+k}$
which is zero by Lemma \ref{lem:no homogeneous element}. Hence, $h'\in\mathbb{C}\langle\sigma_{1},\dots,\sigma_{n}\rangle$
which gives $h\in\mathbb{C}\langle\sigma_{1},\dots,\sigma_{n},\sigma_{n}^{-1}\rangle$.
However, $\mathbb{C}\langle\sigma_{1},\dots,\sigma_{n},\sigma_{n}^{-1}\rangle\cap\mathcal{O}_{q}\big(M_{n}(\mathbb{C})\big)=\mathbb{C}\langle\sigma_{1},\dots,\sigma_{n}\rangle$
as they are the subalgebras of cocommutative elements in $\mathcal{O}_{q}\big(GL_{n}(\mathbb{C})\big)$
and $\mathcal{O}_{q}\big(M_{n}(\mathbb{C})\big)$.
\end{proof}

\section{Case of $\mathcal{O}_{q}\big(SL_{2}(\mathbb{C})\big)$\label{sec:Case-of n=00003D2}}

In this section, we prove Theorem \ref{thm:The-centralizer-of-the-trace}
for $\mathcal{O}_{q}\big(SL_{2}(\mathbb{C})\big)$ which is the base
step of the induction that we use in the proof of the general case.
In fact, in the induction step we will show the statement for $\mathcal{O}_{q}\big(M_{n}(\mathbb{C})\big)$
and not for $\mathcal{O}_{q}\big(SL_{n}(\mathbb{C})\big)$ but in
the light of Proposition \ref{prop:implications} these are equivalent.
The only reason why we use $SL_{2}$ in this part and not $M_{2}$
is that $\mathcal{O}_{q}\big(SL_{2}(\mathbb{C})\big)$ has fewer elements
(in the sense of Gelfand-Kirillov dimension) so the computations are
shorter.
\begin{prop}
Assume that $q\in\mathbb{C}^{\times}$ is not a root of unity. The
centralizer of $\overline{\sigma}_{1}\in\mathcal{O}_{q}\big(SL_{2}(\mathbb{C})\big)$
is $\mathbb{C}\langle\sigma_{1}\rangle$.
\end{prop}
For simplicity, we will use the notations $a:=x_{1,1}+(\mathrm{det}_{q}-1)$,
$b:=x_{1,2}+(\mathrm{det}_{q}-1)$, $c:=x_{2,1}+(\mathrm{det}_{q}-1)$
and $d:=x_{2,2}+(\mathrm{det}_{q}-1)$ for the generators of $\mathcal{O}_{q}\big(SL_{2}(\mathbb{C})\big)$.
In particular, $\overline{\sigma}_{1}=a+d$.

By Theorem I.7.16. in \cite{BG} we have a basis of $\mathcal{O}_{q}\big(SL_{2}(\mathbb{C})\big)$
consisting of the following elements: 
\[
a^{i}b^{k}c^{l},\ b^{k}c^{l}d^{j},\ b^{k}c^{l}\qquad(i,j\in\mathbb{N}^{+},\, k,l\in\mathbb{N})
\]
We will use the $\mathbb{Z}/2\mathbb{Z}$-grading of $\mathcal{O}_{q}\big(SL_{2}(\mathbb{C})\big)$
defined as $\deg(a^{i}b^{k}c^{l})=i\ \mathrm{mod}\,2$ and $\deg(b^{k}c^{l}d^{j})=j\ \mathrm{mod}\,2$.
Note, that it is not the $\mathbb{Z}/2\mathbb{Z}$-grading that it
inherits from the $\mathbb{Z}$-grading of $\mathcal{O}_{q}\big(M_{2}(\mathbb{C})\big)$
which would be $i+k+l$ and $k+l+j$ modulo $2$. Still, this is a
grading in the sense of graded algebras.
\begin{proof}
First, let us compute the action of $\overline{\sigma}_{1}=a+d$ on
the basis elements:
\begin{eqnarray*}
(a+d)\cdot a^{i}b^{k}c^{l} & = & a^{i+1}b^{k}c^{l}+(1+q^{-1}bc)a^{i-1}b^{k}c^{l}=\\
 & = & a^{i+1}(b^{k}c^{l})+a^{i-1}(b^{k}c^{l}+q^{-2(i-1)-1}b^{k+1}c^{l+1})
\end{eqnarray*}
and similarly,
\begin{eqnarray*}
a^{i}b^{k}c^{l}\cdot(a+d) & = & q^{-(k+l)}a^{i+1}b^{k}c^{l}+q^{k+l}a^{i-1}(1+qbc)b^{k}c^{l}=\\
 & = & a^{i+1}(q^{-(k+l)}b^{k}c^{l})+a^{i-1}(q^{k+l}b^{k}c^{l}+q^{k+l+1}b^{k+1}c^{l+1})
\end{eqnarray*}
Hence, for the commutator, we get 
\begin{eqnarray}
\big[(a+d),a^{i}b^{k}c^{l}\big] & = & a^{i+1}\big((1-q^{-(k+l)})b^{k}c^{l}\big)+\label{eq:commutator n=00003D2}\\
 & + & a^{i-1}\big((1-q^{k+l})b^{k}c^{l}+\nonumber \\
 & + & (q^{-2(i-1)-1}-q^{k+l+1})b^{k+1}c^{l+1}\big)\nonumber 
\end{eqnarray}
By the same computation on $b^{k}c^{l}d^{j}$ and $b^{k}c^{l}$, one
can conclude that
\begin{eqnarray*}
\big[(a+d),b^{k}c^{l}d^{j}\big] & = & \big((q^{-(k+l)}-1)b^{k}c^{l}\big)d^{j+1}+\\
 & + & \big((q^{k+l}-1)b^{k}c^{l}+\\
 & + & (q^{k+l+1}-q^{-2(j-1)-1})b^{k+1}c^{l+1}\big)d^{j-1}
\end{eqnarray*}
\[
\big[(a+d),b^{k}c^{l}\big]=a(1-q^{-(k+l)})b^{k}c^{l}+(q^{-(k+l)}-1)b^{k}c^{l}d
\]
Generally, for a polynomial $p\in\mathbb{C}[t_{1},t_{2}]$ and $i\geq1$:
\begin{eqnarray}
\big[(a+d),a^{i}p(b,c)\big] & = & a^{i+1}\sum_{m}\big((1-q^{-m})p_{m}(b,c)\big)+\label{eq:commutator n=00003D2 a*p(b,c)}\\
 & + & a^{i-1}\bigg(\sum_{m}(1-q^{m})p_{m}(b,c)+\nonumber \\
 & + & (q^{-2(i-1)-1}-q^{m+1})bc\cdot p_{m}(b,c)\bigg)\nonumber 
\end{eqnarray}
where $p_{m}$ is the $m$-th homogeneous component of $p$ with respect
to the $\mathbb{N}$-valued total degree on $\mathbb{C}[t_{1},t_{2}]\cong\mathbb{C}\langle b,c\rangle$.
The analogous computations for $p(b,c)d^{j}$ ($j\geq1$) and $p(b,c)$
give 
\begin{eqnarray}
\big[(a+d),p(b,c)d^{j}\big] & = & \sum_{m}\big((q^{-m}-1)p_{m}(b,c)\big)d^{j+1}+\label{eq:commutator n=00003D2 p(b,c)*d}\\
 & + & \sum_{m}\big((q^{m}-1)p_{m}(b,c)+\nonumber \\
 & + & (q^{m+1}-q^{-2(j-1)-1})bcp_{m}(b,c)\big)d^{j-1}\nonumber 
\end{eqnarray}
\begin{equation}
\big[(a+d),p(b,c)\big]=a\sum_{m}\big((1-q^{-m})p_{m}(b,c)\big)+\sum_{m}\big((q^{-m}-1)p_{m}(b,c)\big)d\label{eq:commutator n=00003D2 p(b,c)}
\end{equation}

To prove the statement, it is enough to show that in each subspace
$\sum_{i=0}^{\alpha}a^{i}\cdot\mathbb{C}\langle b,c\rangle+\sum_{j=0}^{\alpha}\mathbb{C}\langle b,c\rangle\cdot d^{i}\leq\mathcal{O}_{q}\big(SL_{2}(\mathbb{C})\big)$
the space of $\overline{\sigma}_{1}$-centralizing elements has dimension
$\alpha+1$. Indeed, $\sum_{i=0}^{\alpha}\mathbb{C}\overline{\sigma}_{1}^{i}$
has exactly dimension $\alpha+1$ by $\mathbb{C}\langle\overline{\sigma}_{1}\rangle\cong\mathbb{C}[t]$,
and these are $\overline{\sigma}_{1}$-centralizing elements, so then
there cannot be anything else that commutes with $\overline{\sigma}_{1}$. 

Assume that the nonzero element $g$ commutes with $\overline{\sigma}_{1}$
and express $g$ in the above mentioned basis as: 
\[
g=\sum_{i=1}^{\alpha}a^{i}r_{i}+\sum_{j=1}^{\beta}s_{j}d^{j}+u
\]
where $r_{i}$, $s_{j}$ and $u$ are elements of $\mathbb{C}\langle b,c\rangle$,
and $\alpha$ and $\beta$ are the highest powers of $a$ and $d$
appearing in the decomposition (i.e. $r_{\alpha}\neq0$ and $s_{\beta}\neq0$).
We will also write $r_{0}$ or $s_{0}$ for $u$, if it makes the
formula simpler. Since $\overline{\sigma}_{1}$ is a homogeneous element
with respect to the $\mathbb{Z}/2\mathbb{Z}$-grading, we may assume
that $g$ is also homogeneous. 

The proof is split into two cases: if $g$ has degree $0\in\mathbb{Z}/2\mathbb{Z}$
(hence $\alpha$ is even) then we will prove that the constant terms
of the $\frac{\alpha}{2}+1$ polynomials $r_{\alpha},r_{\alpha-2},\dots,r_{2},u\in\mathbb{C}[b,c]$
determine $g$ uniquely, and similarly, if $g\in\mathbb{Z}/2\mathbb{Z}$
has degree $1$ (hence $\alpha$ is odd) then the constant terms of
the $\frac{\alpha+1}{2}$ polynomials $r_{\alpha},r_{\alpha-2},\dots,r_{1}\in\mathbb{C}[b,c]$
also determine $g$ uniquely. This is enough, since then in the even
case, we get $\big(\frac{\alpha}{2}+1\big)+\frac{(\alpha-1)+1}{2}=\alpha+1$
for the dimension of the $\overline{\sigma}_{1}$-centralizing elements
as the sum of dimensions of homogeneous $\overline{\sigma}_{1}$-centralizing
elements in even and odd degrees. Similarly, if $\alpha$ is odd,
it is $\frac{\alpha+1}{2}+\frac{\alpha-1}{2}=\alpha+1$ so it is indeed
enough to prove the above claim.

First, we prove that $r_{\alpha}\in\mathbb{C}\cdot1$ in both cases.
If $\alpha=0$ then $r_{\alpha}=u$ so the $a^{i}b^{k}c^{l}$ terms
in $[a+d,g]$ (decomposed in the monomial basis) are the same as the
$a^{i}b^{k}c^{l}$ terms in $[a+d,u]$ by \ref{eq:commutator n=00003D2 a*p(b,c)},
\ref{eq:commutator n=00003D2 p(b,c)*d} and \ref{eq:commutator n=00003D2 p(b,c)}.
However, by \ref{eq:commutator n=00003D2 p(b,c)}, these terms would
be nonzero if $u\notin\mathbb{C}$. Now, assume that $\alpha\geq1$
and define the subspace 
\[
\mathcal{A}^{d}:=\mathrm{Span}_{\mathbb{C}}(a^{i}b^{k}c^{l},\ b^{k}c^{l}d^{j},\ b^{k}c^{l}\ |\ i\leq d,\ k,l\in\mathbb{N})
\]
for any $d\in\mathbb{N}$. Then, by the fact that $\overline{\sigma}_{1}\mathcal{A}^{\alpha-1}$,
$\mathcal{A}^{\alpha-1}\overline{\sigma}_{1}$, $d\mathcal{A}^{\alpha}$
and $\mathcal{A}^{\alpha}d$ are all contained in $\mathcal{A}^{\alpha}$
(using the defining relations), we have 
\[
\overline{\sigma}_{1}g-g\overline{\sigma}_{1}+\mathcal{A}^{\alpha}\subseteq\overline{\sigma}_{1}(a^{\alpha}r_{\alpha}+\mathcal{A}^{\alpha-1})-(a^{\alpha}r_{\alpha}+\mathcal{A}^{\alpha-1})\overline{\sigma}_{1}+\mathcal{A}^{\alpha}=a\cdot a^{\alpha}r_{\alpha}-a^{\alpha}r_{\alpha}\cdot a+\mathcal{A}^{\alpha}
\]
Moreover, if $r_{\alpha}=\sum\lambda_{k,l}b^{k}c^{l}$ then $a^{\alpha}r_{\alpha}\cdot a=\sum\lambda_{k,l}q^{-k-l}a^{\alpha+1}b^{k}c^{l}$.
Since the elements $a^{\alpha+1}b^{k}c^{l}$ are independent even
modulo $\mathcal{A}^{\alpha}$ by Section 2.3, $a^{\alpha}r_{\alpha}\cdot a$
can agree with $a^{\alpha+1}r_{\alpha}$ modulo $\mathcal{A}^{\alpha}$
only if $\lambda_{k,l}=0$ for all $(k,l)\neq(0,0)$. Therefore, $r_{\alpha}\in\mathbb{C}\cdot1$.

Now, we prove that for all $1\leq i\leq\alpha-1$, $r_{i+1}$ and
the constant term of $r_{i-1}$ determines $r_{i-1}\in\mathbb{C}[b,c]$.
Indeed, by equation \ref{eq:commutator n=00003D2 a*p(b,c)} we have
\begin{eqnarray}
0=\mathrm{Coeff}_{a^{i}}\big([(a+d),g]\big) & = & \sum_{m}\big((1-q^{-m})r_{i-1,m}\big)+\label{eq:coeff}\\
 & + & \sum_{m}(1-q^{m})r_{i+1,m}+\nonumber \\
 & + & \sum_{m}(q^{-2(i-1)-1}-q^{m+1})bc\cdot r_{i+1,m}\nonumber 
\end{eqnarray}
where and $r_{i,m}$ is the $m$-th homogeneous term of $r_{i}\in\mathbb{C}[b,c]$
and $\mathrm{Coeff}_{a^{i}}$ stands for the element in $\mathbb{C}[b,c]$
such that $a^{i}\cdot\mathrm{Coeff}_{a^{i}}(x)$ is a summand of $x$
when it is decomposed in the monomial basis. The degree $k$ part
of the right hand side is 
\[
(1-q^{-k})r_{i-1,k}+(1-q^{k})r_{i+1,k}+(q^{1-2i}-q^{k-1})bc\cdot r_{i+1,k-2}\qquad\textrm{if }k\geq2
\]

\[
(1-q^{-1})r_{i-1,1}+(1-q^{1})r_{i+1,1}\qquad\textrm{if }k=1
\]
for all $1\leq i\leq\alpha-1$. Hence $r_{i+1}$ determines $r_{i-1}$
(using that $q$ is not a root of unity) except for the constant term
$r_{i-1,0}$ which has zero coefficient in \ref{eq:coeff} for all
$k$. 

We prove that $\deg s_{j+1}\leq\deg s_{j-1}-2$ for all $j\geq1$
where $\deg$ stands for the total degree of $\mathbb{C}[b,c]$. Analogously
to \ref{eq:coeff}, one can deduce the following by \ref{eq:commutator n=00003D2 p(b,c)*d}:
\begin{eqnarray*}
0=\mathrm{Coeff}_{d^{j}}\big([(a+d),g]\big) & = & \sum_{m}\big((q^{-m}-1)s_{j-1,m}\big)+\\
 & + & \sum_{m}(q^{m}-1)s_{j+1,m}\\
 & + & \sum_{m}(q^{m+1}-q^{-2(j-1)-1})bc\cdot s_{j+1,m}
\end{eqnarray*}
The degree $k$ part of the right hand side is

\begin{equation}
(q^{-k}-1)s_{j-1,k}+(q^{k}-1)s_{j+1,k}+(q^{k-1}-q^{1-2j})bc\cdot s_{j+1,k-2}\qquad\textrm{if }k\geq2\label{eq:determining s}
\end{equation}

\[
(q^{-1}-1)s_{j-1,1}+(q^{1}-1)s_{j+1,1}\qquad\textrm{if }k=1
\]
for all $1\leq j\leq\beta-1$. Note that $q^{k-1}-q^{1-2j}=0$ can
never happen for $k\geq2$. If $s_{j+1}=0$ then the statement is
empty. If $s_{j+1}\neq0$ then for $k=2+\deg s_{j+1}\geq2$, we have
$s_{j+1,k}=0$ but $s_{j+1,k-2}=s_{j+1,\deg s_{j+1}}\neq0$ hence
\ref{eq:determining s} gives $s_{j-1,k}\neq0$. So $\deg s_{j+1}\leq\deg s_{j-1}-2$.

Now, assume that $\alpha$ is even. By the previous paragraphs, the
scalars $r_{\alpha}$, $r_{\alpha-2,0},\dots,$ $r_{2,0}$ and $u_{0}$
determine all the polynomials $r_{\alpha},r_{\alpha-2},r_{\alpha-4},\dots,r_{2}$
and $u$. We prove that they also determine the $s_{j}$'s. Starting
from $u=s_{0}$ one can obtain $s_{j+1}$ by $s_{j-1}$. Indeed, since
$\deg s_{j+1}\leq\deg s_{j-1}-2$, if $\deg s_{j-1}\leq1$ then $s_{j+1}=0$,
and similarly, for $k=\deg s_{j-1}\geq2$ we have $s_{j+1,k-1}=0$
and \ref{eq:determining s} gives $(q^{-k}-1)s_{j-1,k}=-(q^{k-1}-q^{1-2j})bc\cdot s_{j+1,k-2}$.
Then, recursively for $k$, if $s_{j-1,k}$ and $s_{j+1,k}$ are given,
by \ref{eq:determining s} they determine $s_{j+1,k-2}$, using that
$q$ is not a root of unity.

If $\alpha$ is odd, then by \ref{eq:commutator n=00003D2 a*p(b,c)}
one can obtain the following for the summand of $[(a+d),g]$ that
does not contain $a$ and $d$ when decomposed in the given basis:
\begin{eqnarray*}
0=\mathrm{Coeff}_{1}\big([(a+d),g]\big) & = & \sum_{m}\Big((1-q^{m})r_{1,m}+(q-q^{m+1})bc\cdot r_{1,m}+\\
 & + & (q^{m}-1)s_{1,m}+(q^{m+1}-q)bc\cdot s_{1,m}\Big)
\end{eqnarray*}
The homogeneous components of degree $k$ are
\begin{equation}
(1-q^{k})r_{1,k}+(q-q^{k-1})bc\cdot r_{i+1,k-2}+(q^{k}-1)s_{1,k}+(q^{k+1}-q)bc\cdot s_{1,k-2}\qquad\textrm{if }k\geq2\label{eq:determining s, odd}
\end{equation}

\[
(1-q)r_{1,1}+(q-1)s_{1,1}\qquad\textrm{if }k=1
\]
Hence, $r_{\alpha},r_{\alpha-2,0},\dots,r_{1,0}$ determine not only
$r_{i}$ for $1\leq i\leq\alpha$ but also $s_{1}$ by \ref{eq:determining s, odd}
applied for $k=\deg s_{1}+2$ and the same recursive argument as in
the even case. Then, similarly, $s_{j+1}$ is unique by $s_{j-1}$
for all $2\leq j\leq\beta-1$ and the statement follows.
\end{proof}

\section{Proof of the main result\label{sec:Proof}}

In \cite{DL1}, to verify that the subalgebra of cocommutative elements
in $A_{n}:=\mathcal{O}_{q}\big(M_{n}(\mathbb{C})\big)$ is generated
by the $\sigma_{i}$'s, they proved that the natural surjection 
\[
\eta:\mathcal{O}_{q}\big(M_{n}(\mathbb{C})\big)\to\mathbb{C}[t_{1},\dots,t_{n}]\qquad x_{i,j}\mapsto\delta_{i,j}t_{i}
\]
restricted to the subalgebra of cocommutative elements $\mathcal{O}_{q}\big(M_{n}(\mathbb{C})\big)^{\mathrm{coc}}$
is an isomorphism and its image is the subalgebra of symmetric polynomials
$D_{n}^{S_{n}}$ where $D_{n}:=\mathbb{C}[t_{1},\dots,t_{n}]$. We
use the same plan to prove that it is also the centralizer of $\sigma_{1}\in\mathcal{O}_{q}\big(M_{n}(\mathbb{C})\big)$.

For this purpose, we will need the following intermediate quotient
algebra between $A_{n}$ and $D_{n}$:
\[
B_{2,n}:=A_{n}/(x_{1,j},\, x_{i,1}\ |\ 2\leq i,j\leq n)
\]
Let us denote the corresponding natural surjection by $\varphi:A_{n}\to B_{2,n}$.
Since $\mathrm{Ker}\,\eta\subseteq\mathrm{Ker}\,\varphi$ by their
definition, $\eta$ can be factored through $\varphi$. So our setup
is: 
\begin{equation}
\xymatrix{C(\sigma_{1})\ar@{}[r]|\subseteq & A_{n}\ar@{->>}[r]^{\varphi} & B_{2,n}\ar@{->>}[r]^{\delta} & D_{n}}
\label{eq:setup}
\end{equation}
where $\eta=\delta\circ\varphi$ and $C(\sigma_{1})$ denotes the
centralizer of $\sigma_{1}$ in $A_{n}$. The structure of $B_{2,n}$
is quite simple: $B_{2,n}\cong A_{n-1}[t]$ by the map $x_{i,j}\mapsto x_{i-1,j-1}$
for $i,j\geq2$ and $x_{1,1}\mapsto t$. One can check that it is
indeed an isomorphism since $x_{1,1}$ commutes with the elements
of $\mathbb{C}\langle x_{1,1},x_{i,j}\ |\ i,j\geq2\rangle$ modulo
$\mathrm{Ker}\,\varphi$. 

These algebras are $\mathbb{N}$-graded algebras using the total degree
of $A_{n}$, but we can also endow them by a filtration that is not
the corresponding filtration of the grading. Namely, for each $d\in\mathbb{N}$
let $\mathcal{A}^{d}$ be the subspace of $A_{n}$ that is generated
by the monomials in which $x_{1,1}$ appears at most $d$ times, i.e.
it is spanned by the ordered monomials of the form $x_{1,1}^{i}m$
where $i\leq d$ and $m$ is an ordered monomial in the variables
$x_{i,j}$, $(i,j)\neq(1,1)$. One can check that this is indeed a
filtration: they are linear subspaces such that $\cup_{d}\mathcal{A}^{d}=A_{n}$
and $\mathcal{A}^{d}\cdot\mathcal{A}^{e}\subseteq\mathcal{A}^{d\cdot e}$
for all $d,e\in\mathbb{N}$. As $C(\sigma_{1})$ is a subalgebra of
$A_{n}$, we get an induced filtration $\mathcal{C}^{d}=\mathcal{A}^{d}\cap C(\sigma_{1})$
($d\in\mathbb{N}$) on $C(\sigma_{1})$, and similarly, an induced
filtration $\mathcal{B}^{d}:=\varphi(\mathcal{A}^{d})$ ($d\in\mathbb{N}$)
on $B_{2,n}$ and $\mathcal{D}^{d}:=\delta\circ\varphi(\mathcal{A}^{d})$
($d\in\mathbb{N})$ on $D_{n}$.

\medskip{}

\begin{proof}[Proof of Theorem \ref{thm:The-centralizer-of-the-trace}]
 We prove the statement by induction on $n$. The statement is verified
for $\mathcal{O}_{q}\big(SL_{2}(\mathbb{C})\big)$ in Section \ref{sec:Case-of n=00003D2}
so by Proposition \ref{prop:implications} the case $n=2$ is proved.
Now, assume that $n\geq3$. We shall prove that
\begin{itemize}
\item $(\delta\circ\varphi)|_{C(\sigma_{1})}:C(\sigma_{1})\to D_{n}$ is
injective, and
\item the image $(\delta\circ\varphi)\big(C(\sigma_{1})\big)$ is in $D_{n}^{S_{n}}$.
\end{itemize}
This means that the restriction of $\delta\circ\varphi$ to $C(\sigma_{1})$
yields an isomorphism with $D_{n}^{S_{n}}$, since by \cite{DL1},
$C(\sigma_{1})\ni\sigma_{i}$ for all $i=1,\dots,n$ and $\delta\circ\varphi$
restricts to an isomorphism between $\mathbb{C}\langle\sigma_{1},\dots,\sigma_{n}\rangle$
and $D_{n}^{S_{n}}$. 

\textbf{First part, step 1:} First, we show that it is enough to prove
that $\mathrm{gr}(\delta\circ\varphi)$ restricted to $\mathrm{gr}\big(C(\sigma_{1})\big)$
is injective to get the injectivity of $\delta\circ\varphi$ on $C(\sigma_{1})$.
Apply $\mathrm{gr}$ to the filtered algebras in our setup presented
in Diagram \ref{eq:setup}. It gives 
\begin{equation}
\xymatrix{\mathrm{gr}\big(C(\sigma_{1})\big)\ar@{}[r]|\subseteq & \mathrm{gr}(A_{n})\ar@{->>}[r]^{\mathrm{gr}(\varphi)} & \mathrm{gr}(B_{2,n})\ar@{->>}[r]^{\mathrm{gr}(\delta)} & \mathrm{gr}(D_{n})}
\label{eq:setup2}
\end{equation}
The surjectivity of the maps follow by $\varphi(\mathcal{A}^{d})=\mathcal{B}^{d}$
and $\delta(\mathcal{B}^{d})=\mathcal{D}^{d}$. Assuming that $\mathrm{gr}(\delta\circ\varphi)$
restricted to $\mathrm{gr}\big(C(\sigma_{1})\big)$ is injective,
we get the injectivity of $(\delta\circ\varphi)|_{\mathcal{C}^{0}}$,
moreover, we can apply an induction on $d$ using the $5$-lemma in
the following commutative diagram of vector spaces for all $d\geq1$:
\[
\xymatrix{0\ar[r] & \mathcal{C}^{d-1}\ar@{^{(}->}[d]^{\delta\circ\varphi\mid_{\mathcal{C}^{d-1}}}\ar[r] & \mathcal{C}^{d}\ar[d]^{\delta\circ\varphi\mid_{\mathcal{C}^{d}}}\ar[r] & \mathcal{C}^{d}/\mathcal{C}^{d-1}\ar@{^{(}->}[d]^{\mathrm{gr}(\delta\circ\varphi|_{C(\sigma_{1})})_{d}}\ar[r] & 0\\
0\ar[r] & \mathcal{D}^{d-1}\ar[r] & \mathcal{D}^{d}\ar[r] & \mathcal{D}^{d}/\mathcal{D}^{d-1}\ar[r] & 0
}
\]
where the rows are exact by definition and $\mathrm{gr}(\delta\circ\varphi|_{C(\sigma_{1})})_{d}$
and $\delta\circ\varphi\mid_{\mathcal{C}^{d-1}}$ are injective by
the assumption and the induction hypothesis. Therefore, $\delta\circ\varphi$
is injective on $\cup_{d}\mathcal{C}^{d}=C(\sigma_{1})$. 

Notice that $B_{2,n}$ and $D_{n}$ are not only filtered by the $\varphi(x_{1,1})$
and $t_{1}$ degrees but they are also graded as $B_{2,n}\cong A_{n-1}[t]$
and $D_{n}\cong D_{n-1}[t]$ by $t_{1}\mapsto t$ and $t_{i}\mapsto t_{i-1}\in D_{n-1}$
($i\geq2$). Hence, we will use the natural identifications of graded
algebras $B_{2,n}\cong\mathrm{gr}(B_{2,n})$ and $\mathrm{gr}(D_{n})\cong D_{n}$
(and so $\mathrm{gr}(\delta)$ is just $\delta$).

\textbf{Step 2:} We prove that the image of the map $\mathrm{gr}(\varphi)$
restricted to $\mathrm{gr}\big(C(\sigma_{1})\big)$ is in $C\big(\varphi(\sigma_{1})\big)\subseteq B_{2,n}$.
Here, $C\big(\varphi(\sigma_{1})\big)$ is a graded subalgebra of
$B_{2,n}$ since $\varphi(\sigma_{1})$ is a sum of a central element
$\varphi(x_{1,1})$ and of the elements $\varphi(x_{2,2}),\dots,\varphi(x_{n,n})$
(that are homogeneous of degree zero) so $C\big(\varphi(\sigma_{1})\big)=C\big(\varphi(x_{2,2}+\dots+x_{n,n})\big)$
is homogeneous. The proof of this step is clear: For an $h\in\mathcal{C}^{d}\subseteq\mathcal{A}^{d}$
we have $0=\varphi\big([\sigma_{1},h]\big)=\big[\varphi(\sigma_{1}),\varphi(h)\big]$,
hence $\mathrm{gr}(\varphi)(h+\mathcal{C}^{d-1})\in C\big(\varphi(\sigma_{1})\big)$. 

\textbf{Step 3:} We prove the injectivity of $\mathrm{gr}(\delta)$
restricted to $C\big(\varphi(\sigma_{1})\big)$ by the induction.
First, note that $C\big(\varphi(\sigma_{1})\big)\cong C_{A_{n-1}}(\sigma_{1})[t]$
using the isomorphism $B_{2,n}\cong A_{n-1}[t]$. Then, by the induction
hypothesis, 
\[
C_{A_{n-1}}(\sigma_{1})=\mathbb{C}\big\langle\sigma_{1}(A_{n-1}),\dots,\sigma_{n-1}(A_{n-1})\big\rangle
\]
Therefore, $C\big(\varphi(\sigma_{1})\big)=\mathbb{C}\langle\sigma_{1}(B_{2,n}),\dots,\sigma_{n-1}(B_{2,n}),\varphi(x_{1,1})\rangle$
where $\sigma_{i}(B_{2,n})$ is defined as the image of $\sigma_{i}(A_{n-1})$
under the above mentioned isomorphism. For these elements, we have
$\delta\big(\sigma_{i}(B_{2,n})\big)=s_{i}(t_{2},\dots,t_{n})$ where
$s_{i}(t_{2},\dots,t_{n})$ is the $i$-th elementary symmetric polynomial
in the variables $t_{2},\dots,t_{n}$. Hence, $\delta$ is indeed
injective by the fundamental theorem of symmetric polynomials. Now,
it is enough to prove the injectivity of $\mathrm{gr}(\varphi)$ restricted
to $C(\sigma_{1})$ to get the injectivity of $\delta\circ\varphi$
by Step 1 and 2.

\textbf{Step 4:} For $\mathrm{ad}\sigma_{1}:A_{n}\to A_{n}$, $h\mapsto[\sigma_{1},h]$,
we have $C(\sigma_{1})=\mathrm{Ker}(\mathrm{ad}\sigma_{1})$ by definition.
Although $\mathrm{ad}\sigma_{1}$ is not a morphism of algebras but
a derivation of degree 1, we can still take 
\[
\mathrm{Ker}\big(\mathrm{gr}(\mathrm{ad}\sigma_{1})\big):=\big\{(h_{d})_{d\in\mathbb{N}}\in\mathrm{gr}(A_{n})\ |\ \sigma_{1}h_{d}-h_{d}\sigma_{1}+\mathcal{A}^{d}=0\in\mathcal{A}^{d+1}/\mathcal{A}^{d}\big\}
\]
where $\mathrm{gr}(\mathrm{ad}\sigma_{1})$ is understood as a map
of graded vector spaces. Then, we can extend Diagram \ref{eq:setup2}
as:
\[
\xymatrix{\mathrm{gr}(A_{n})\ar@{->>}[r]^{\mathrm{gr}(\varphi)} & B_{2,n}\ar@{->>}[r]^{\mathrm{gr}(\delta)} & D_{n}\\
\mathrm{Ker}\big(\mathrm{gr}(\mathrm{ad}\sigma_{1})\big)\ar@{}[u]|\bigcup & C_{B_{2,n}}(\sigma_{1})\ar@{}[u]|\bigcup\\
\mathrm{gr}\big(C(\sigma_{1})\big)\ar@{}[u]|\bigcup\ar[ru]
}
\]
Naturally, $\mathrm{gr}\big(C(\sigma_{1})\big)\subseteq\mathrm{Ker}\big(\mathrm{gr}(\mathrm{ad}\sigma_{1})\big)$
since $\sigma_{1}h_{d}-h_{d}\sigma_{1}=0\in A_{n}$ implies $\sigma_{1}h_{d}-h_{d}\sigma_{1}\in\mathcal{A}^{d}$.

We give an explicit description of $\mathrm{Ker}\big(\mathrm{gr}(\mathrm{ad}\sigma_{1})\big)$.
Observe that 
\[
\mathrm{gr}(A_{n})\cong\bigoplus_{d\in\mathbb{N}}y^{d}\mathbb{C}\langle x_{i,j}\ |\ (i,j)\neq(1,1)\rangle
\]
 where $y$, the image of $x_{1,1}$, commutes with every $x_{i,j}$
for $2\leq i,j\leq n$ and $q$-commutes with $x_{1,j}$ and $x_{i,1}$
for all $i,j\geq2$. Indeed, by the monomial basis of $A_{n}$ (see
Section \ref{sec:Prerequisites}) we get the direct sum decomposition,
moreover, the only defining relations involving $x_{1,1}$ are $x_{1,1}x_{1,j}=qx_{1,j}x_{1,1}$,
$x_{1,1}x_{i,1}=qx_{i,1}x_{1,1}$ and $x_{1,1}x_{i,j}=x_{i,j}x_{1,1}+(q-q^{-1})x_{i,1}x_{1,j}$
that reduce to $q$-commutativity of $y$ and commutativity of $y$
with the appropriate elements. The argument also gives that the image
of the monomial basis of $A_{n}$ is a monomial basis in $\mathrm{gr}(A_{n})$. 

In particular, we get that 
\[
\mathrm{Ker}\big(\mathrm{gr}(\mathrm{ad}\sigma_{1})\big)\cong\bigoplus_{d\in\mathbb{N}}y^{d}\mathbb{C}\langle x_{i,j}\ |\ 2\leq i,j\leq n\rangle
\]
 by the same isomorphism. Indeed, for an element $x_{1,1}^{d}m\in\mathcal{A}^{d}$
where $m$ is an ordered monomial in the variables $x_{i,j}$ ($(i,j)\neq(1,1)$),
we have
\[
\mathrm{gr}(\mathrm{ad}\sigma_{1})(x_{1,1}^{d}m+\mathcal{A}^{d-1})=x_{1,1}\cdot x_{1,1}^{d}m-x_{1,1}^{d}m\cdot x_{1,1}+\mathcal{A}^{d}
\]
since $x_{i,i}\cdot\mathcal{A}^{d}\subseteq\mathcal{A}^{d}$ and $\mathcal{A}^{d}\cdot x_{i,i}\subseteq\mathcal{A}^{d}$
for all $i\geq2$. Then, by the above mentioned $q$-commutativity
relations, we get $(1-q^{-c(m)})x_{1,1}^{d+1}m+\mathcal{A}^{d}$ where
$c(m)$ stands for the sum of exponents of the $x_{1,j}$'s and $x_{i,1}$'s
($2\leq i,j\leq n$) appearing in $m$. The result is a monomial basis
element in $\mathcal{A}^{d+1}/\mathcal{A}^{d}\subseteq\mathrm{gr}(A_{n})$.
For different monomials $x_{1,1}^{d}m$ and $x_{1,1}^{d'}m'$ we get
different monomials $x_{1,1}^{d+1}m$ and $x_{1,1}^{d'+1}m'$ so $\mathrm{gr}(\mathrm{ad}\sigma_{1})$
is diagonal in the monomial basis of $\mathrm{gr}(A_{n})$ with the
scalars $(1-q^{-c(m)})$. Hence, its kernel is $\{x_{11}^{d}m+\mathcal{A}^{d-1}\ |\ d\in\mathbb{N},\ c(m)=0\}$
since $q$ is not a root of unity, as we stated. Therefore, we get
that $\mathrm{Ker}\big(\mathrm{gr}(\mathrm{ad}\sigma_{1})\big)\cong A_{n-1}[t]$
using $y\mapsto t$ and $x_{i,j}\mapsto x_{i-1,j-1}$ since $y$ commutes
with every $x_{i,j}$ for $2\leq i,j\leq n$.

Now, the injectivity part of the theorem follows: the isomorphisms
$B_{2,n}\cong A_{n-1}[t]$ and $\mathrm{Ker}\big(\mathrm{gr}(\mathrm{ad}\sigma_{1})\big)\cong A_{n-1}[t]$
established in step 4 are compatible, meaning that $\mathrm{gr}(\varphi)$
composed with them on the appropriate sides is $\mathrm{id}_{A_{n-1}[t]}$.
In particular, $\mathrm{gr}(\varphi)$ restricted to $\mathrm{gr}\big(C(\sigma_{1})\big)\subseteq\mathrm{Ker}\big(\mathrm{gr}(\mathrm{ad}\sigma_{1})\big)$
is injective. By step 3, $\mathrm{gr}(\delta)$ restricted to $C\big(\varphi(\sigma_{1})\big)$
is also injective, so the composition $\delta\circ\mathrm{gr}(\varphi)=\mathrm{gr}(\delta\circ\varphi)$
is injective as well, using step 2. By step 1, this means that $\delta\circ\varphi$
is injective.

\textbf{Second part:} To prove $\eta\big(C(\sigma_{1})\big)\subseteq D_{n}^{S_{n}}$,
consider the following commutative diagram:
\[
\xymatrix{A_{n}\ar@{->>}[r]^{\varphi} & B_{2,n}\ar@{->>}[r]^{\delta} & D_{n}\\
C(\sigma_{1})\ar@{}[u]|\bigcup\ar[r] & C\big(\varphi(\sigma_{1})\big)\ar@{}[u]|\bigcup\ar[r] & D_{n}^{S_{n-1}}\ar@{}[u]|\bigcup
}
\]
where $S_{n-1}$ acts on $D_{n}$ by permuting $t_{2},\dots,t_{n}$.
The diagram implicitly states that $\varphi\big(C(\sigma_{1})\big)\subseteq C\big(\varphi(\sigma_{1})\big)$
(which is clear) and that $\delta\big(C\big(\varphi(\sigma_{1})\big)\big)\subseteq D_{n}^{S_{n-1}}$.
The latter follows by the induction hypothesis for $n-1$: it gives
that $C\big(\varphi(\sigma_{1})\big)=\mathbb{C}\langle\sigma_{1}(B_{2,n}),\dots,\sigma_{n-1}(B_{2,n}),\varphi(x_{1,1})\rangle$
by $B_{2,n}\cong A_{n-1}[t]$ and since $\delta(\varphi(x_{1,1}))=t_{1}$
and $\delta\big(\sigma_{i}(B_{2,n})\big)=s_{i}(t_{2},\dots,t_{n})$,
the $i$-th elementary symmetric polynomial in the variables $t_{2},\dots,t_{n}$,
we get that $(\delta\circ\varphi)\big(C(\sigma_{1})\big)$ is symmetric
in $t_{2},\dots,t_{n}$. 

To prove symmetry in $t_{1},\dots,t_{n-1}$ too, consider the isomorphism
$\gamma:\mathcal{O}_{q}\big(M_{n}(\mathbb{C})\big)\cong\mathcal{O}_{q^{-1}}\big(M_{n}(\mathbb{C})\big)$
given by $x_{i,j}\leftrightarrow x'_{n+1-i,n+1-j}$ where $x'_{i,j}$
denotes the variables in $\mathcal{O}_{q^{-1}}\big(M_{n}(\mathbb{C})\big)$.
This is indeed an isomorphism: interpreted in the free algebra it
maps the defining relations of $\mathcal{O}_{q}\big(M_{n}(\mathbb{C})\big)$
to the defining relations of $\mathcal{O}_{q^{-1}}\big(M_{n}(\mathbb{C})\big)$.
It also maps $\sigma_{1}\in\mathcal{O}_{q}\big(M_{n}(\mathbb{C})\big)$
into the $\sigma_{1}$ of $\mathcal{O}_{q^{-1}}\big(M_{n}(\mathbb{C})\big)$
denoted by $\sigma_{1}'$. Moreover, $\overline{\gamma}\circ\eta=\eta'\circ\gamma$
where $\overline{\gamma}:D_{n}\to D_{n}$, $t_{i}\mapsto t_{n+1-i}$
($i=1,\dots,n$) and $\eta':\mathcal{O}_{q^{-1}}\big(M_{n}(\mathbb{C})\big)\to\mathbb{C}[t_{1},\dots,t_{n}]$,
$x'_{i,j}\mapsto t_{i}\delta_{i,j}$ is the $\eta$($=\delta\circ\varphi$)
of $\mathcal{O}_{q^{-1}}\big(M_{n}(\mathbb{C})\big)$. Hence, $(\overline{\gamma}\circ\eta)\big(C(\sigma_{1})\big)=\eta'\big(C(\sigma'_{1})\big)$
as $C(\sigma_{1}')$ is symmetric under $\overline{\gamma}$. Applying
the previous argument on $\mathcal{O}_{q^{-1}}\big(M_{n}(\mathbb{C})\big)$
gives that $\eta'\big(C(\sigma'_{1})\big)\subseteq D_{n}^{S_{n-1}}$
where $S_{n-1}$ still acts by permuting $t_{2},\dots,t_{n}$. Hence,
$\eta\big(C(\sigma_{1})\big)$ is symmetric in $t_{1},\dots,t_{n-1}$
too so we got that $\eta\big(C(\sigma_{1})\big)$ is symmetric in
all the variables $t_{1},\dots,t_{n}$ by $n\geq3$. 
\end{proof}

\medskip{}

\begin{rem}
In fact, the proof of injectivity of $\eta$ is valid in the case
$n=2$ too, but the symmetry argument used to prove $\eta\big(C(\sigma_{1})\big)\subseteq\mathbb{C}[t_{1},\dots,t_{n}]^{S_{n}}$
does not give anything if $n=2$. That is why we had to start the
induction at $n=2$ instead of $n=1$.
\end{rem}

\begin{rem}
\label{rem:final remark automorphism}As it is discussed in \cite{DL1},
the set of cocommutative elements in $\mathcal{O}_{q}\big(GL_{n}(\mathbb{C})\big)$
is the ring of invariants under the right coaction
\[
\alpha:\mathcal{O}_{q}\big(GL_{n}(\mathbb{C})\big)\to\mathcal{O}_{q}\big(GL_{n}(\mathbb{C})\big)\otimes\mathcal{O}_{q}\big(GL_{n}(\mathbb{C})\big)
\]
\[
a\mapsto\sum a_{(2)}\otimes a_{(3)}S(a_{(1)})
\]
where we use Sweedler's notation. Although this coaction does not
agree with the right adjoint coaction 
\[
a\mapsto\sum a_{(2)}\otimes S(a_{(1)})a_{(3)}
\]
of the Hopf algebra $\mathcal{O}_{q}\big(GL_{n}(\mathbb{C})\big)$
(that is also mentioned in the referred article) but they differ only
by the automorphism $S^{2}$. Hence, by Theorem \ref{thm:The-subalgebra},
the invariants of the right adjoint coaction also form a maximal commutative
subalgebra.

We get other maximal commutative subalgebras by applying automorphisms
of the algebras $\mathcal{O}_{q}\big(GL_{n}(\mathbb{C})\big)$, $\mathcal{O}_{q}\big(M_{n}(\mathbb{C})\big)$
or $\mathcal{O}_{q}\big(SL_{n}(\mathbb{C})\big)$, though they do
not have many automorphisms: it is proved in \cite{Y2} establishing
a conjecture stated in \cite{LL} that the automorphism group of $\mathcal{O}_{q}\big(M_{n}(\mathbb{C})\big)$
is generated by the transpose operation on the variables and a torus
that acts by rescaling the variables $x_{i,j}\mapsto c_{i}d_{j}x_{i,j}$
($c_{i},d_{j}\in\mathbb{C}^{\times}$).\end{rem}

\end{document}